\documentclass[a4paper,11pt]{amsart}

\usepackage{amsmath, amssymb}
\usepackage[english]{babel}
  \usepackage{paralist}
  \usepackage{graphics} %% add this and next lines if pictures should be in esp format
  \usepackage{epsfig} %For pictures: screened artwork should be set up with an 85 or 100 line screen
\usepackage{graphicx}
\usepackage{epstopdf}%This is to transfer .eps figure to .pdf figure; please compile your paper using PDFLeTex or PDFTeXify.
 \usepackage[colorlinks=true]{hyperref}
 \usepackage{array}
\usepackage[latin1]{inputenc}
\usepackage{xcolor}
\usepackage{esint}
\usepackage{latexsym,amsfonts}
\usepackage{amsthm}
\usepackage{verbatim}

\usepackage{vmargin}

\setmarginsrb{30mm}{20mm}{30mm}{20mm}{10mm}{10mm}{10mm}{10mm}

\usepackage{array}
\newtheorem{theorem}{Theorem}[section]
\newtheorem{cor}{Corollary}[section]

\newtheorem{lem}{Lemma}[section]

\newtheorem{prop}{Proposition}[section]

\theoremstyle{definition}
\newtheorem{defn}{Definition}[section]

\newtheorem{rem}{Remark}[section]
\newtheorem*{theorem*}{Theorem}

\newtheorem{ex}{Example}[section]

\numberwithin{equation}{section}

\newcommand{\R}{\mathbb R}
\newcommand{\N}{\mathbb N}

\newcommand{\Z}{\mathbb Z}

\newcommand{\bigint}{\begin{picture}(10,10)
\put(-1,2){\line(1,0){10}}
\end{picture}\kern-14pt\int}

\def\Xint#1{\mathchoice
   {\XXint\displaystyle\textstyle{#1}}%
   {\XXint\textstyle\scriptstyle{#1}}%
   {\XXint\scriptstyle\scriptscriptstyle{#1}}%
   {\XXint\scriptscriptstyle\scriptscriptstyle{#1}}%
   \!\int}
\def\XXint#1#2#3{{\setbox0=\hbox{$#1{#2#3}{\int}$}
     \vcenter{\hbox{$#2#3$}}\kern-.5\wd0}}

\def\dashint{\Xint-}
\definecolor{blau}{rgb}{0.1,0.0,0.9}
\definecolor{violet}{rgb}{0.54, 0.17, 0.89}
\newcommand{\blue}{\color{blau}}

\newcommand{\kom}[1]{}
\renewcommand{\kom}[1]{{\bf \blue /#1/}}

\newcounter{komcounter}
\numberwithin{komcounter}{section}

\title{THE LIOUVILLE THEOREM FOR DISCRETE SYMMETRIC AVERAGING OPERATORS}

\author{Tomasz Adamowicz{\small$^1$}}
\address{T.A.: Institute of Mathematics, Polish Academy of Sciences,
\'Sniadeckich 8, Warsaw, 00-656, Poland.\/}
\email{tadamowi@impan.pl}

\author{Jos\'e G. Llorente{\small$^2$}}
\address{J.G.L.: Departamento de An\'alisis Matem\'atico y Matem\'atica Aplicada and Instituto de Matem\'atica Interdisciplinar (IMI). Facultad de Matem\'aticas, Universidad Complutense de Madrid. Plaza de Ciencias 3, 28040 Madrid, Spain.\/}
\email{josgon20@ucm.es}

\begin{document}

%\date{\today}

\footnotetext[1]{T. Adamowicz was supported by a grant of National Science Center, Poland (NCN),
 UMO-2017/25/B/ST1/01955.}
\footnotetext[2]{J. G. Llorente was partially supported by Ministerio de Econom\'ia y Competitividad grant MTM2017-85666-P.}
\begin{abstract}
We introduce averaging operators on lattices $\Z^d$ and study the Liouville property for functions satisfying mean value properties associated to such operators. This framework encloses  discrete harmonic, $p$-harmonic, $\infty$-harmonic and the so-called game $p$-harmonic functions. Our approach provides an elementary alternative proof of the Liouville Theorem for positive $p$-harmonic functions on $\Z^d$.
\end{abstract}

\maketitle

\section{Introduction}  The interest in discrete  harmonic functions goes back to the nineteenth century, closely related to the study of electrical networks and random walks. Nowadays, Discrete Potential Theory is an active field, with connections and applications to different areas of pure and applied mathematics such as discretization of PDEs, stochastic games, machine learning and image processing, among others (see \cite{C, ETT, GP} and references therein). Hereafter we will restrict to a particular graph: the lattice $\Z^d$. Let $e_1,...,e_d$ be the canonical basis of $\R^d$ and, for notational convenience, set $e_{j+d} = -e_j$ for $j = 1,...,d$. For each $x\in \Z^d$ the points  $n_j(x) := x + e_j$, $j = 1,...,2d$ are the $2d$ neighbours of $x$. Then $u : \Z^d \to \R$ is \textit{harmonic} if
\begin{equation}\label{hgraph}
u(x) = \frac{1}{2d}\sum_{j=1}^{2d} u(x + e_j)
\end{equation}
for each $x\in \Z^d$. Observe that \eqref{hgraph} is the discrete counterpart  of the classical Mean Value Property for harmonic functions on $\Z^d$. Suppose now that $\mu_j >0$ for $j= 1,...,2d$ and that $\mu_1 + \cdots + \mu_{2d} = 1$ (we say in this case that $ \mu = (\mu_1,..., \mu_{2d})$ is an \textit{admissible probability distribution}). Then $u: \Z^d \to \R$ is said to be \textit{weighted harmonic} with respect to $\mu $ if
\begin{equation}\label{whgraph}
u(x) = \sum_{j=1}^{2d} \mu_j u(x + e_j)
\end{equation}
for each $x\in \Z^d$. In particular, weighted harmonic functions with $ \mu = (1/2d, ..., 1/2d)$ are just harmonic functions.

  If $1 < p < \infty$, a function $u : \Z^d \to \R$ is $p$-\textit{harmonic} if, for each $x\in \Z^d$,
 \begin{equation}\label{phgraph}
 \sum_{j=1}^{2d}  (u(x + e_j ) - u(x)) |u(x + e_j) - u(x) |^{p-2} = 0
 \end{equation}
Also, $u$ is \textit{weighted} $p$-\textit{harmonic} with respect to $\mu = (\mu_1,..., \mu_{2d})$ if
\begin{equation}\label{wpgraph}
\sum_{j=1}^{2d} \mu_j (u(x + e_j ) - u(x)) |u(x + e_j) - u(x) |^{p-2} = 0
\end{equation}

It turns out that  $p$-harmonic functions minimize the discrete  $p$-energy in the same way that harmonic functions  minimize the discrete $2$-energy. Note that \eqref{phgraph} is the discrete version of the $p$-\textit{Laplace} equation
$$
{\rm div} (|\nabla u|^{p-2}\nabla u ) = 0
$$
whose solutions (in a weak sense) are also called $p$-harmonic. In both settings: in the discrete and the continuous one, harmonic functions are, of course, recovered when $p=2$.

\

The classical Mean Value Property for harmonic functions is the key of the fruitful interplay between Potential Theory, Probability and PDEs. In recent years much work has been devoted to understand the probabilistic framework associated to some relevant nonlinear differential operators, such as the $p$-Laplacian and the $\infty$-Laplacian, in terms of certain nonlinear mean value properties. A relevant example (either in the Euclidean setting or in a metric measure space) is the following functional equation
 \begin{equation}\label{pharmo}
 u(x)=\frac{\alpha}{2}(\sup_{B_x}u+\inf_{B_x}u)+(1-\alpha)\dashint_{B_x}u \, d \mu.
 \end{equation}
 where $u$ is a (continuous) function in the closure of a bounded  given domain $\Omega$ in $\R^n$ or in  a metric measure space $(X, d, \mu)$, $B_x$ is a ball centered at $x$  contained in $\Omega$  and $\alpha \in[0,1]$. Note that when $\alpha = 0$ then we retrieve harmonic functions (see also \cite{AGG, AL}), whereas for $\alpha = 1$ we obtain the so-called \textit{harmonious} functions in $\Omega$, see \cite{LGA}. It turns out that, in the Euclidean setting and in the range $p \geq 2$, the functional equation \eqref{pharmo}  with $ \alpha = \alpha (p,d) = \frac{p-2}{p+2d}$ can be understood as a \textit{Dynamic Programming Principle} for a stochastic game related to the $p$-Laplacian, see \cite{PSSW, PS, MPR2, MPR, BR} and references therein for further details. Functions satisfying \eqref{pharmo} with $\alpha = \alpha (p,d)$ have been called $p$-harmonious, see \cite{AL,MPR}. It is then natural to consider discrete versions of equation \eqref{pharmo}. Let  $\alpha \in [0,1]$. We will say that $u: \Z^d \to \R$ is $\alpha$-\textit{harmonious} if
 \begin{equation}\label{alphahar}
 u(x) = \frac{\alpha}{2} \big ( \max_{j=1,...,2d} u(x+ e_j) + \min_{j=1,...,2d}u(x+e_j) \big ) + \frac{1- \alpha}{2d} \sum_{j=1}^{2d} u(x + e_j)
 \end{equation}
for each $x \in \Z^d$. Solutions of \eqref{alphahar} have also been called \textit{game} $p$-\textit{harmonic} by some authors because of their connections to stochastic games related to the $p$-Laplacian. In order to simplify notation, we prefer to call them  $\alpha$-harmonious and, thus, eliminate any reference to $p$ in the definition. The case $\alpha = 1$ deserves a special treatment, e.g. due to its importance in the Lipschitz extension problem and in the tug-of-war games, see~\cite{LGA, ETT, BEJ, MRT}. We say that $u: \Z^d \to \R$ is $\infty$-\textit{harmonic} if
\begin{equation}\label{inftyhar}
u(x) = \frac{1}{2} \big ( \max_{j=1,...,2d} u(x+ e_j) + \min_{j=1,...,2d}u(x+e_j) \big )
\end{equation}
for each $x\in \Z^d$.

In order to consider a more general framework, we need to introduce the concept of averaging operator.

 \begin{defn}\label{def:avoper}
  A function $F: \R^{2d} \to \R$ is an  \textit{averaging operator} if it satisfies the following structural assumptions:
\begin{itemize}
\item[i)]  $F$ is strictly increasing in each variable, that is, for each $j = 1,...,2d$ the (one variable) function
$$
t\to F(t_1,\ldots,t_{j-1}, t, t_{j+1},\ldots,t_{2d})
$$
is strictly increasing.  \vspace{0.12cm}
\item[ii)] $F(t_1 + c ,\ldots,t_{2d} + c) = c + F(t_1,\ldots,t_{2d})$ and $F(ct_1,\ldots,ct_{2d}) = c F(t_1,\ldots,t_{2d})$ for any $c \in \R$.
	
	%\item[iii)] For each $j \in \{ 1, 2, \ldots, 2d \}$ it holds that $F(0,\ldots,0,1,0,\ldots,0) >0$, where $1$ is located at position $j$.
	%We define the positive constant $\lambda$, such that $\lambda$ is the minimum of those $F(0,\ldots,0,1,0,\ldots,0)$ for $j =1,\ldots,2d$.

%\item[iv)] Let $\Gamma_1,\ldots,\Gamma_d \geq 1$ and $a_1,\ldots,a_d\in \{-1,1\}$ be such that
 %$$
 %F(\Gamma_1^{a_1}, \ldots,\Gamma_d^{a_d}, \Gamma_1^{-a_1},\ldots,\Gamma_d^{-a_d} ) = 1.
 %$$
 %Then $\Gamma_1 = \Gamma_2 = \ldots = \Gamma_d = 1$.
\end{itemize}
\end{defn}

We introduce now the class of functions satisfying a mean value property associated to an averaging operator.
\begin{defn}\label{def:Fharm}
Let $F: \R^{2d} \to \R$ be an averaging operator. A function $u : \Z^d \to \R$ is said to be $F$-\textit{harmonic} on $\Z^d$ if
\begin{equation}\label{F-harm}
u(x) = F \big ( u(x+e_1), \ldots,u(x + e_{2d}) \big )
\end{equation}
for each $x\in \Z^d$.
\end{defn}

In Section \ref{sec:ex} we will see that weighted harmonic, $p$-harmonic and $\alpha$-harmonious functions for $\alpha \in [0,1)$  are examples of $F$-harmonic functions on $\Z^d$ for appropriate averaging operators satisfying Definition~\ref{def:avoper}. See also Section \ref{sec:someprop} for the basic properties of averaging operators and $F$-harmonic functions on $\Z^d$.

\

Averaging operators in regular directed trees were introduced in \cite{ARY} and  developed in \cite{KLLW} with the purpose of studying set theoretical properties of nonlinear harmonic measures on the boundary of the tree (see also \cite{DMR} for further results in the same direction). The interest in the discrete setting can be explained from several perspectives. First, if a noncompact Riemannian manifold $M$ has controlled local geometry, then the $p$-harmonic parabolicity of $M$ can be characterized by the parabolicity of appropriate discrete structures, the so-called $\epsilon$-nets of $M$. Moreover, $p$-parabolicity implies the strong Liouville property for $p$-harmonic functions, see e.g.~\cite[Corollary 2.14]{HS2}, and the existence of nonconstant $p$-harmonic functions on $M$ can be equivalently expressed for $\epsilon$-nets, see \cite{HS1, HS2} and their references. On the other hand, it has recently been shown that, under certain regularity assumptions on a bounded domain $\Omega \subset \R^n $,  solutions of the $p$-harmonic Dirichlet Problem in $\Omega $ can be approximated by appropriate rescaled solutions of equation \eqref{pharmo}, see \cite{MPR2, AL2}. The $\alpha$-harmonious equation \eqref{alphahar} arises then as a natural discretization of \eqref{pharmo}. Finally, nonlinear discrete models have shown to be very useful to solve some potential theoretic problems in the, more difficult, continuous setting. One such example is the lack of additivity of the $p$-harmonic measure at zero level, see \cite{KLLW, LMW}.

\

The main topic that we address in this note is the \textit{strong Liouville property} for $F$-harmonic functions on $\Z^d$. More precisely,  we seek conditions on the averaging operator $F$ implying that nonnegative $F$-harmonic functions on $\Z^d$ are constant. A classical theorem by  B\^ocher (1903) and Picard (1923) asserts that nonnegative harmonic functions in $\R^n$ must be constant. The B\^ocher-Picard result is usually known as the Liouville theorem, because of the analogy with the classical result for entire holomorphic functions. As for the discrete setting, the fact that nonnegative harmonic functions on $\Z^d$ are constant is implicit in the early works of Le Roux \cite{LR} and Capoulade \cite{C}; see also the excellent paper by Duffin \cite{Du}. In these three works, the authors observe that the use of  explicit formulas for the Poisson discrete kernel in squares implies a Harnack inequality, from which the strong Liouville property follows in a standard way.

\

The following two concepts will play a relevant role in the paper.

\begin{defn}\label{def:symm}
Let $F: \R^{2d} \to \R$ be an averaging operator and $(\mu_1 ,..., \mu_{2d})$ an admissible probability distribution.
\begin{enumerate}
\item[i)]  We say that $F$ is \textit{symmetric} if, for any $t_1,...,t_{2d}\in \R$ and each $j\in \{ 1,...,d\}$,
$$
F(t_1,...,t_{2d}) = F(t_{\pi_j(1)},...,t_{\pi_j(2d)})
$$
where $\pi_j$ is the transposition of the set $\{ 1,...,2d \}$ which interchanges indexes $j$ and $j+d$, that is $\pi_j (j) = j+d$, $\pi_j (j+d) = j$ and $\pi_j (k) = k$ if $k\in \{1,...,2d \} \setminus \{j, j+d \} $. \vspace{0.1cm}
\item[ii)] We say that  $(\mu_1 ,..., \mu_{2d})$ is  \textit{symmetric} if $\mu_{j+d} = \mu_j$ for each $j\in \{ 1,...,d \}$.
    \end{enumerate}
\end{defn}
Observe that if $(\mu_1 ,..., \mu_{2d})$ is a symmetric probability distribution then $F(t_1,...,t_{2d}) = \mu_1 t_1 +...+ \mu_{2d}t_{2d}$ is a symmetric averaging operator.

\begin{defn}\label{def:charac}
A function $u : \Z^d \to \R$ is said to be a \textit{multiplicative character} on $\Z^d$ if there are positive constants $A_1,...,A_d$ such that
\begin{equation}\label{charact}
u(x_1,...,x_d ) = A_1^{x_1}\cdots A_d^{x_d}
\end{equation}

\end{defn}

In the context of harmonic functions on abelian groups there is an abstract approach to the Liouville property linking the notions of symmetry, multiplicative characters and minimal functions. Let $\mu = (\mu_1,..., \mu_{2d})$ be an admissible probability distribution,  $\mathcal{H}$ the cone of all nonnegative weighted harmonic functions on $\Z^d$ with respect to $\mu$ and $\mathcal{H}_{0} = \{ u \in \mathcal{H} \, :  u(0,...,0) = 1 \} $. Note that $\mathcal{H}_{0}$ is convex and compact (with the topology of pointwise convergence; see Proposition \ref{compact} at this respect). We say that $u\in \mathcal{H}$ is  \textit{minimal} if, whenever $u,v \in \mathcal{H}$ such that $0 \leq v \leq u$ then there exists $\lambda \geq 0$ such that $v = \lambda u$. One can invoke the Krein-Milman theorem to deduce the existence of a minimal function  $u \in \mathcal{H}_{0}$ and it is easy to see that $u$ must be a character of the form \eqref{charact}, where $A_1,...,A_d$ satisfy
\begin{equation}\label{compati}
\sum_{j=1}^d ( \mu_j A_j + \mu_{j+d}A_j^{-1} ) = 1,
\end{equation}
see \cite{S}. Furthermore, the distribution $\mu = (\mu_1,..., \mu_{2d} )$ is  symmetric if and only if $A_1 = \cdots = A_d = 1$. Since, by the Martin theory, any function in $\mathcal{H}$ can be represented in terms of minimal functions (\cite{S}), this program would imply that if $\mu$ is symmetric then $\mathcal{H}_{0}$ reduces to constant function equal to $1$ and, therefore, nonnegative weighted harmonic functions on $\Z^d$ with respect to a symmetric distribution must be constant. This is a particular instance of a more general result by Choquet-Deny and Doob-Snell-Williamson, see~\cite{CD, DSW} and also \cite[Theorem 7.1]{S}. We refer to \cite{Del} for another account of symmetric distributions in the context of harmonic Harnack inequality on general graphs, 
to~\cite{BLMS} for yet another recent development about the harmonic Liouville property on $\Z^2$ and~\cite{T} for an expository article about harmonic functions on $\Z^2$. In summary, by the above discussion we infer that in the linear weighted case and for an admissible probability distribution $\mu$, it holds that \emph{the distribution $\mu$ is symmetric if and only if the (strong) Liouville property holds for harmonic functions on $\Z^d$ with respect to the distribution $\mu$.}

\

As for the nonlinear setting, the strong Liouville property for $p$-harmonic functions in $\R^n$ and $\Z^d$ has been obtained, in both cases, from the much stronger Harnack inequality (see \cite{HKM} for the continuous case and and \cite{HS2}, where the authors consider general graphs with certain geometrical restrictions). It is worth remarking that the proof in \cite{HS2} is highly non-elementary and relies on establishing discrete counterparts of  nonlinear PDEs techniques, with Caccioppoli inequalities, Sobolev-Poincar\'e inequalities and discrete versions of the John-Nirenberg theorem as key tools. Other versions of the Liouville theorem for $p$-harmonic functions or quasiminimizers of the $p$-energy have been obtained in the context of  manifolds and metric measure spaces, see~\cite{BBS1, BBS2} and references therein. In these two settings, Liouville theorems are related to the notion of parabolicity and conditions on the underlying space allowing for the existence of nontrivial $p$-harmonic functions.

\

The main motivation of our work was a comment in \cite{HS2} where the authors suggested the convenience of finding a more elementary and direct approach to the Harnack and Liouville properties for $p$-harmonic functions on graphs. Given an averaging operator $F: \R^{2d} \to \R$, the concept of minimal functions still makes sense but, unfortunately, we cannot invoke the Krein-Milman theorem. By using elementary methods, we circumvent this difficulty and  obtain a positive $F$-harmonic function on $\Z^d$ which is also a multiplicative character and plays, in some sense, the role of a minimal function in this context. If, in addition, $F$ is symmetric then such a character must be constant, implying the strong Liouville property for $F$. The following is the main result of our work.

\

\begin{theorem}[Strong Liouville property] \label{thm:main}
Let  $F: \R^{2d} \to \R$ be a symmetric averaging operator. Then any nonnegative $F$-harmonic function on $\Z^d$ is constant.
\end{theorem}

We stress the fact that symmetry is an essential hypothesis in Theorem \ref{thm:main}, even for weighted harmonic functions on $\Z$, see Remark \ref{symmexa} and the above discussion.

Our method provides, in particular, an alternative and more elementary proof of the strong Liouville property for discrete $p$-harmonic functions and weighted $p$-harmonic functions (with respect to symmetric distributions) on $\Z^d$, cf. Section~\ref{sec:ex}.

\begin{cor}\label{cor6.2} Let $1 < p < \infty$ and let $\mu = (\mu_1,..., \mu_d )$ be a symmetric probability distribution. Then any nonnegative weighted $p$-harmonic function on $\Z^d$ with respect to $\mu$ is constant. In particular, nonnegative $p$-harmonic functions on $\Z^d$ are constant.
\end{cor}

Moreover, our approach also applies to $\alpha$-harmonious functions.

\begin{cor}\label{cor6.1} Let $0 \leq \alpha < 1$. Any nonnegative $\alpha$-harmonious function on $\Z^d$ is constant. In particular, if $\alpha=0$, then we retrieve the strong Liouville property for discrete harmonic functions.
\end{cor}
This corollary follows immediately from Theorem~\ref{thm:main} and the fact that for $0\leq \alpha < 1$ the $\alpha$-harmonious functions are $F$-harmonic for a symmetric averaging operator $F$, cf. Section~\ref{sec:ex}.

The case $\alpha = 1$ deserves a special treatment because the corresponding operator does not satisfy the requirements of  Definition \ref{def:avoper}, see \ref{ex:inftylap}.

\begin{theorem}\label{prop:inftylap}
Any nonnegative $\infty$-harmonic function on $\Z^d$ is constant.
\end{theorem}

\vspace{0.12cm}

{\bf Acknowledgements:} This research was initiated  when the second author was visiting IMPAN, as a guest of the Simons semester ``Geometry and analysis in function and mapping theory on Euclidean and metric measure spaces"  supported by the grant 346300 for IMPAN from the Simons Foundation and the matching 2015-2019 Polish MNiSW fund. He wishes to thank IMPAN for hospitality and support. Moreover, the authors would like to express their gratitude to the referees for insightful comments, questions and suggestions that have contributed to improve the paper significantly.

\section{Examples}\label{sec:ex}

We begin with examples of operators satisfying Definition~\ref{def:avoper} and in the later part of this section we will discuss operators failing some of the properties i) or ii).

We remind that $\mu = ( \mu_1 ,..., \mu_{2d})$ is an admissible probability distribution if $\mu_j >0$ for $j=1,...,2d$ and  $\mu_1 +...+ \mu_{2d} =1$. The distribution is symmetric if $\mu_{j+d} = \mu_j$ for $j= 1,...,d$.

In the first two examples we discuss the discrete weighted $p$-Laplacians $F^{\mu}_p$ with the case $p=2$ corresponding to the Laplacian.

\begin{ex}[the discrete weighted Laplacian] Let $\mu = (\mu_1,...,\mu_{2d})$ be an admissible probability distribution. We define
$$
F^{\mu}_2 (t_1,\ldots,t_{2d}) = \sum_{j=1}^{2d} \mu_j t_j
$$

Then it is easy to check that $F$ satisfies properties i) and ii) in Definition \ref{def:avoper} so $F^{\mu}_2$ is an averaging operator. If, in addition, $\mu$ is  symmetric then $F^{\mu}_2$ is a symmetric averaging operator. Notice that for the uniform distribution $\mu_1=\ldots=\mu_{2d}=(2d)^{-1}$ we retrieve the classical {\bf discrete Laplacian} with $F_2(t_1,\ldots,t_{2d}) = (2d)^{-1} \sum_{j=1}^{2d} t_j$, also cf.~\eqref{hgraph}.

\end{ex}

Next example illustrates that weighted $p$-harmonic functions are $F$-harmonic for appropriate averaging operators.
\begin{ex} [the discrete weighted $p$-Laplacian]  \label{ex-pharm} Let $1 < p < \infty$ and $\mu = (\mu_1, \cdots, \mu_{2d})$ an admissible probability distribution.  Given  $t_1,\ldots,t_{2d} \in \R$, let us define $t$ as the unique real number defined by the implicit identity
\begin{equation}\label{eq:p-harm}
\sum_{j=1}^{2d} \mu_j(t-t_j) |t - t_j|^{p-2} = 0.
\end{equation}
We put
\begin{equation}\label{phar}
F^{\mu}_p(t_1,\ldots,t_{2d}) = t.
\end{equation}

The existence and uniqueness of such a $t$ are consequences of the following observations. First we note that, for each $a\in \R$ the (one variable) function $s \to (s-a) |s-a|^{p-2}$ is strictly increasing in $\R$. Now choose  $t_1,\ldots,t_{2d} \in \R$ and define $\psi : \R \to \R$ by
    \begin{equation}\label{psi}
    \psi (s) =  \sum_{j=1}^{2d} \mu_j|s - t_j|^{p}.
    \end{equation}
Then $\psi$ is differentiable and
   $$
    \psi' (s) = p\sum_{j=1}^{2d} \mu_j(s-t_j)|s-t_j|^{p-2}.
    $$
Since $s \to (s-t_j) |s-t_j|^{p-2}$ is strictly increasing in $\R$ and weights $\mu_j$ are positive, it follows that $\psi$ is strictly convex in $\R$ and it attains a unique minimum at, say, $t\in \R$. Note that $t$ is the unique solution of $\psi '(t) = 0 $, which is equivalent to the $p$-harmonic equation~\eqref{eq:p-harm}. This proves the existence and uniqueness of $F^{\mu}_p (t_1,\ldots,t_{2d})$.

Similar arguments show that $F^{\mu}_p$ is strictly increasing in each variable, cf. property i) of Definition~\ref{def:avoper}. Indeed, let $t_1, T_1, t_2, \ldots,t_{2d} \in \R$ be such that $t_1 < T_1$ and denote by $t := F^{\mu}_p (t_1,t_2,\ldots,t_{2d})$ and $T:= F^{\mu}_p(T_1,t_2,\ldots,t_{2d})$. Let $\psi$ be as in (\ref{psi}). Since $t_1 < T_1$ we have
$$
(t_1 - T)|t_1 -T|^{p-2} < (T_1 -T)|T_1 -T|^{p-2}
$$
or, equivalently, $(T - T_1)|T-T_1|^{p-2} < (T -t_1)|T-t_1|^{p-2}$. Then
\begin{eqnarray*}
\psi '(T) & = & \mu_1(T-t_1)|T-t_1|^{p-2} + \sum_{j=2}^{2d}\mu_{j}(T-t_j)|T-t_j|^{p-2}  \\ & > &  \mu_1 (T- T_1)|T-T_1|^{p-2} +  \sum_{j=1}^{2d}\mu_j (T-t_j)|T-t_j|^{p-2} = 0
\end{eqnarray*}
that is, $\psi' (T) >0$, which implies that $t < T$ according to the preceding comments. This proves $F^{\mu}_p (t_1, t_2,\ldots,t_{2d}) < F^{\mu}_p(T_1, t_2,\ldots,t_{2d})$ as claimed.

It is straightforward to verify that property ii) holds as well. If, in addition, $\mu$ is a symmetric distribution then $F^{\mu}_p$ is a symmetric averaging operator.

As in the case of Laplacian in Example 3.1, for the uniform distribution $\mu_1=\ldots=\mu_{2d}= (2d)^{-1}$ we retrieve the {\bf classical discrete $p$-Laplacian} with $F_p(t_1,\ldots,t_{2d}) = t$, where $\sum_{j=1}^{2d} (t-t_j) |t - t_j|^{p-2} = 0$.

\end{ex}

\begin{ex} [$\alpha$-harmonious averaging operators]\label{ex-pharmonious} Suppose that $0 \leq \alpha <1$. Define
\begin{equation}\label{alpha}
F(t_1,\ldots,t_{2d}) = \frac{\alpha }{2}\big ( \max_{1 \leq j
\leq 2d} t_j + \min_{1 \leq j \leq 2d}t_j \big ) + \frac{1 -
\alpha }{2d} \sum_{j=1}^{2d} t_j.
\end{equation}
Then, one can easily verify properties i) and ii) in Definition \ref{def:avoper}. Note that the restriction $\alpha < 1$ is needed in order to deduce that $F$ is strictly increasing in each variable. The averaging operator $F$ is clearly symmetric.

\end{ex}

\begin{ex}[the discrete $\infty$-Laplacian]\label{ex:inftylap}
An important particular case of the $\alpha$-harmonious operator comes up when $\alpha \equiv 1$, in which case
\[
F_\infty(t_1,\ldots, t_{2d}) = \frac12\left(\max_{1 \leq i \leq 2d} t_i + \min_{1 \leq i \leq 2d}t_i \right)
\]
and the corresponding operator is called the discrete $\infty$-Laplacian. It is easy to check that $F_\infty$ satisfies property ii). As for i), $F$ is non-decreasing with respect to each variable but fails to be strictly increasing. For example, if $d=2$, then  $F_{\infty}(1,0,0,0) = F_{\infty}(1,1,0,0)$.

\end{ex}

We close this section with some examples of operators failing Definition~\ref{def:avoper}.
\begin{ex}
Recall the definition of the discrete $1$-Laplacian for $u:\Z^d\to \R$:
  \[
   \Delta_1u(x)={\rm median}\{u(n_1(x))-u(x),\ldots, u(n_{2d}(x))-u(x)\},
  \]
 where the median is defined as in e.g. Definition 6 in~\cite{MOS}. Namely, the median of the set $\{x_1,\ldots, x_k\}$ is found by arranging all the numbers $x_i$, for $i=1,\ldots, k$ from the lowest value to highest one and selecting the middle one. If $k$ is even, the median is defined to be the mean of the two middle values.

 \noindent Hence, the related averaging operator $F_1:\R^{2d}\to \R$ can be defined as follows: $F_{1}(t_1,\ldots,t_{2d}):={\rm median}\{t_1,\ldots, t_{2d}\}$. It is easy to see that property ii) holds. However, property i) fails as, for instance, $F_1(1,2,3)=F_1(1.5,2,3)=2$. Moreover, $F_1(0,\ldots,1,\ldots,0)=0$.
\end{ex}

\begin{ex}
 Consider the discrete positive eikonal equation (see Section 3.1 in~\cite{MOS}) and the related $F$ operator $F:\R^{2d}\to \R$ defined as follows
 \[
  F^+(x):=\max_{1\leq i \leq 2d} t_i-1.
 \]
 Property ii) holds for $F^+$ but $F^+(1,\ldots,0)=F^+(0,\ldots,1)=0$ giving that property i) fails. Also,  $F^+(0,\ldots,1,\ldots,0)= 0$. Similar discussion holds for the discrete negative eikonal operator $F^-$.

\end{ex}

\section{Basic properties of averaging operators and $F$-harmonic functions on $\Z^d$}\label{sec:someprop}

This section is devoted to establishing certain properties of averaging operators and related $F$-harmonic functions that are needed to prove the main result in Section~\ref{sec:Liouv}. Recall that if $\{ e_1, \ldots,e_d \}$ is the canonical basis of $\R^d$ then we set  $e_{j+d} = -e_j$ for $j=1,...,d$.
The following proposition collects some basic consequences of Definition~\ref{def:avoper}.
\begin{prop}\label{conseq}
Let $F: \R^{2d} \to \R$ be an averaging operator, according to Definition~\ref{def:avoper}. 
\begin{enumerate}
\item $F(0,\ldots,0) = 0$ and $F(1,\ldots,1) = 1$. \vspace{0.12cm}
\item Let
\begin{equation}\label{deflambda}
 \lambda := \min_{j=1,...,2d} F(0,...,0,1,0,...,0)
\end{equation}
where  $1$ is located at position $j$. Then $0 < \lambda < 1$. We call $\lambda$ the \textbf{ellipticity constant} of $F$. 
\vspace{0.12cm}
\item If $t_1, \ldots , t_{2d} \in \R$ are not all equal then
$$
\min\{t_1,\ldots,t_{2d}\} < F (t_1,\ldots,t_{2d}) < \max \{ t_1,\ldots,t_{2d} \}
$$
\end{enumerate}
\end{prop}

\begin{proof}
$F(0, \ldots , 0) = 0$ follows from property ii) in  Definition~\ref{def:avoper} taking  $c= 0$.  From property i), $F(1,\ldots,1) = F(1+0,\ldots,1+0) = 1 +F(0,\ldots,0) = 1$ so $F(1,\ldots,1) = 1$. That $ 0 < \lambda < 1 $ is consequence of $F(0, \ldots , 0) = 0$, $F(1,\ldots,1) = 1$ and the fact that $F$ is strictly increasing in each variable. Finally, if $t_1,\ldots,t_{2d}$ are not all equal and $t := \max \{t_1,\ldots,t_{2d}\}$ then, by properties i),ii) in Definition \ref{def:avoper} and (1) above we get $F(t_1,\ldots,t_{2d}) < F(t,\ldots,t) = tF(1,\ldots,1) = t$ and a similar argument gives the lower estimate.
\end{proof}

%\begin{rem}\label{ellip}
%We will refer to the constant $\lambda $ in part (2) of Proposition \ref{conseq} as the \textit{ellipticity constant} of $F$.
%\end{rem}

\begin{rem}\label{compar}
Part (3) of Proposition \ref{conseq} directly implies the following observation: if  $F: \R^{2d} \to \R$ is an averaging operator, $t_1' \leq t_1,\ldots,t_{2d}' \leq t_{2d}$ and $F(t_1,\ldots,t_{2d}) = F(t_1',\ldots,t_{2d}')$ then  $t_j = t_j'$ for all $j = 1,\ldots,2d$. This result will be repeatedly used in the paper.
\end{rem}

We obtain now the following modulus of continuity estimate for operators satisfying Definition~\ref{def:avoper}.

\begin{prop}\label{propv} Let $F: \R^{2d} \to \R$ be an averaging operator as in Definition~\ref{def:avoper}. Then, for all $(t_1,\ldots,t_{2d}), (t_1',\ldots,t_{2d}')\in \R^{2d}$ we have that
\begin{equation}\label{prop-v}
 |F(t_1,\ldots,t_{2d}) - F(t_1',\ldots,t_{2d}')| \leq \max\{|t_1-t_1'|,\ldots, |t_{2d}-t_{2d}'| \}.
\end{equation}
In particular, $F$ is continuous in $\R^{2d}$ and
$$
 |F(t_1,\ldots,t_{i-1},t_{i}, t_{i+1},\ldots,t_{2d}) - F(t_1,\ldots,t_{i-1},t_{i}', t_{i+1},\ldots,t_{2d})| \leq |t_i-t_i'|.
$$
\end{prop}	

\begin{proof}
 Set $\delta:=\max\{|t_1-t_1'|,\ldots, |t_{2d}-t_{2d}'| \}$. Then, for each $i=1,\ldots, 2d$, it holds that $t_i'\leq t_i+\delta$ and, by properties i) and ii), we get
 \[
  F(t_1',\ldots,t_{2d}')\leq F(t_1+\delta,\ldots,t_{2d}+\delta)=F(t_1,\ldots,t_{2d})+\delta.
 \]
 Analogously, we obtain that $ F(t_1,\ldots,t_{2d})\leq F(t_1',\ldots,t_{2d}')+\delta$ and~\eqref{prop-v} follows.
\end{proof}

The following proposition collects the main properties of symmetric averaging operators that will be needed later (recall Definition \ref{def:symm}).

\begin{prop}\label{symme}
Let $F: \R^{2d} \to \R$ be a symmetric averaging operator. Then
\begin{enumerate}
\item $F(t_1,\ldots, t_{d}, t_{d+1},\ldots, t_{2d}) = F(t_{d+1},\ldots, t_{2d}, t_1,\ldots, t_{d} )$, for $t_1 , \ldots , t_{2d} \in \R$. \vspace{0.12cm}
\item $F(t_1,\ldots, t_{d}, -t_1,\ldots, -t_{d}) = 0$ for $t_1,\ldots, t_{d} \in \R$. \vspace{0.12cm}
\item  Suppose that $A_1 , \ldots, A_d >0$ and $F(A_1 , \ldots, A_d, A_1^{-1}, \ldots , A_d^{-1} ) = 1$. Then
$$
A_1 = \cdots = A_d = 1.
$$
\end{enumerate}
\end{prop}

\begin{proof}
Part (1) is a consequence of the fact that the permutation
$$
(1, \ldots , 2d) \to (d+1, \ldots, 2d, 1, \ldots , d)
$$
is the composition of the $d$ transpositions interchanging $j, j+d$ for $j =1, \ldots, d$. To see (2) note that, by property ii) in Definition~\ref{def:avoper} and (1), it follows
$$
-F(t_1, \ldots, t_d, -t_1 , \ldots, - t_d) = F(-t_1, \ldots, - t_d, t_1,\ldots, t_d) = F(t_1, \ldots , t_d , -t_1 , \ldots, -t_d)
$$
implying $F(t_1, \ldots , t_d , -t_1 , \ldots, -t_d) = 0$. As for part (3), set $B_j = A_j$ if $A_j \geq 1$ and $B_j = A_j^{-1}$ if $0 < A_j < 1$. Then $B_j \geq 1$ for each $j = 1, \ldots, d$ and, by the assumptions and by the symmetry of $F$,  $F(B_1, \ldots, B_d, B_1^{-1}, \ldots, B_d^{-1}) = 1$. Therefore,
\begin{eqnarray*}
0 & = &  F(B_1 -1, \ldots, B_d -1, B_1^{-1} -1, \ldots, B_d^{-1}-1) \\
& = & F \big ( B_1 -1 , \ldots , B_d -1 , -\frac{B_1 -1}{B_1}, \ldots - \frac{B_d -1}{B_d}   \big ) \\
& \geq & F ( B_1 -1 , \ldots , B_d -1 , - (B_1 -1), \ldots - (B_d -1) )  = 0
\end{eqnarray*}
where the first identity follows from property ii) in Definition~\ref{def:avoper}, the inequality is consequence of the fact that $B_j \geq 1$ together with property i) in Definition~\ref{def:avoper} and the last identity follows from part (2) above. Since $B_j ^{-1} -1 \geq -(B_j -1)$, Remark~\ref{compar} implies that $B_j^{-1} -1 = - (B_j -1)$ and, consequently, $B_j = 1$ for $j= 1, \ldots, d$.
\end{proof}

\

The rest of the section is devoted to the main properties of $F$-harmonic functions that will be used in the proof of the main result. If $x\in \Z^d$, we remind that  $n_j(x) = x + e_j$ are the $2d$ neighbours of $x$ on $\Z^d$, for $j=1,...,2d$. We start by the simple invariance property which follows from Definition~\ref{def:Fharm} and property ii) in Definition~\ref{def:avoper}.

\begin{prop}\label{prop:perm-harm}  Let  $F: \R^{2d} \to \R$ be an averaging operator and $u : \Z^d \to \R$ be $F$-harmonic. Then $v(x) = cu(x)$ and $w(x)=u(x+a)$ are also $F$-harmonic for any $c\in \R$ and $a\in \Z^d$.
\end{prop}	
	
%\begin{proof}
%The assertion for $v$ follows immediately from Definition~\ref{def:Fharm} and the second part of property ii) in Definition~\ref{def:avoper}. Similarly, in order to show that $w(x)=u(x+a)$ is $F$-harmonic, we note the following:
%%Let $x\in \Z^d$ and let $n_1 (x),\ldots,n_{2d}(x)$ be the neighbours of $x$. From the specific form of the transformation $T$ it follows that
%%$$
%%\{ T(n_1 (x)), \ldots, T(n_{2d}(x))\} = \{ n_1 (T(x)) ,\ldots, n_{2d}(T(x))  \}.
%%$$
%%Then, from the permutation-invariance of $F$ and $F$-harmonicity of $u$, it follows that
%\begin{align*}
%w(x)  &=  u(x+a) = F\left( u(n_1 (x+a)),\ldots, u(n_{2d}(x+a)) \right )\\
% & = F \left (u(n_1 (x)+a),\ldots, u(n_{2d}(x)+a) \right) = F \left( w(n_1 (x)),\ldots,w(n_{2d}(x)) \right),
%\end{align*}	
%and so $w$ is $F$-harmonic.
%\end{proof}	

We  remind that  $\lambda$ denotes the ellipticity constant of $F$ given by formula~\eqref{deflambda}.

\begin{prop}[Local Harnack inequality]\label{localH} Let  $F: \R^{2d} \to \R$ be an averaging operator and $u : \Z^d \to \R$ be $F$-harmonic. Then, for any two neighbouring points $x\sim y$ with $x,y\in \Z^d$,
\begin{equation}\label{harnack}
 u(x) \leq \frac{1}{\lambda}u(y)
\end{equation}
where $\lambda$ is the ellipticity constant of $F$.
\end{prop}	
\begin{proof}
Suppose, for simplicity that $\displaystyle t = \max_{i=1,\ldots,2d} u(n_i(x)) = u(n_1 (x))$. Then, by property i), Remark~\ref{compar} and property (2) in Proposition~\ref{conseq}, it holds that
$$
u(x)= F\big(u(n_1(x)),\ldots,u(n_{2d}(x))\big) \geq F(t,0,\ldots,0) \geq t\lambda
$$
and the Harnack estimate follows.
\end{proof}

The following are direct consequences of the local Harnack inequality.

\begin{cor}\label{minimum}
Let $F: \R^{2d} \to \R$ be an averaging operator and $u : \Z^d \to \R$ be $F$-harmonic. If $u \geq 0$ and $u(a) = 0$ for some $a\in \Z^d$, then $u \equiv 0 $.   	
\end{cor}

%\begin{proof} Suppose, on the contrary, that $u$ is positive at some point. Then, by Proposition~\ref{prop5.1}, it holds that $u$ is positive everywhere in $\Z^d$, contradicting assumption that $u(a)=0$ at some $a\in \Z^d$.
%\end{proof}
\begin{cor}\label{prop:loc-bound}(Local boundedness) Let $F: \R^{2d} \to \R$ be an averaging operator and $u : \Z^d \to \R$ be a nonnegative $F$-harmonic function on $\Z^d$ such that $u(0,\ldots,0) = 1$. Then
\begin{equation}\label{bound}	
u(x) \leq \Big ( \frac{1}{\lambda}\Big ) ^{\|x\|} 	
\end{equation}	
for all $x = (x_1,\ldots,x_d)\in \Z^d$, where $\|x\| = |x_1| +\cdots+|x_d|$ and $\lambda$ is the ellipticity constant of $F$.
\end{cor}

%\begin{proof}
%Let $t_i = u(n_i (0,\ldots,0))$. Then
%$$
%1 = F(t_1,\ldots,t_{2d}) \geq F(t_1,0,\ldots,0) = t_1 F(1,0,\ldots,0) \geq \lambda t_1
%$$	
%Therefore $t_1 \leq 1/\lambda$ and, analogously, $t_i \leq 1/\lambda$ for $i=1,2,\ldots,2d $. This shows that \eqref{bound} holds if $\|x\|=1$. The result follows by iterating the same argument.
%\end{proof}

The following result says that the family of $F$-harmonic functions is closed under pointwise convergence. The proof is the direct consequence of the continuity of $F$, see Proposition \ref{propv}.

\begin{prop}\label{prop:conv}
 Let $F: \R^{2d} \to \R$ be an averaging operator and $(u_k)$ a sequence of $F$-harmonic functions on $\Z^d$ such that
$$
\lim_k u_k (x) = u(x)
$$
for each $x\in \Z^d$. Then $u$ is $F$-harmonic.
\end{prop}	
%\begin{proof} If $x\in \Z^d$ then $u_k (x) = F\big(u_k (n_1 (x)),...,u_k (n_{2d}(x))\big)$. Since $F$ is continuous (Proposition \ref{propv}), by taking limits when $k\to \infty$ we get
%$$
%u(x) = F\big(u(n_1(x)), \ldots,u(n_{2d}(x))\big),
%$$
%and so $u$ is $F$-harmonic.
%\end{proof}

\

For a fixed averaging operator $F: \R^{2d} \to \R$ we define
\begin{equation}\label{H0}
\mathcal{H}_0 = \{ u : \Z^d \to \R :  u \, \text{is} \,  F\text{-harmonic}, \, u\geq 0, \, u(0,\ldots,0) = 1   \}.
\end{equation}

\

 The following is a compactness result for the class $\mathcal{H}_0$ reminiscent of the classical Harnack convergence theorem for sequences of nonnegative harmonic functions.

\begin{prop}\label{compact}
Let $F: \R^{2d} \to \R$ be an averaging operator. Then each sequence  $ \{ u_n \} \subset\mathcal{H}_0$ contains a subsequence $\{u_{n_k} \} $ such that $\{ u_{n_k} \} \to u$ pointwise on $\Z^d$ for some $u \in \mathcal{H}_0$.
\end{prop}

\begin{proof}
By Corollary \ref{prop:loc-bound}, the class $\mathcal{H}_0$ is locally bounded so $\{ u_n \}$ is pointwise bounded. Since $\Z^d$ is countable, a standard diagonal argument shows that there exists a subsequence $\{u_{n_k} \}$ so that $u_{n_k} \to u$ pointwise and, by Proposition \ref{prop:conv}, $u \in \mathcal{H}_0$.
\end{proof}

\section{Existence of characters}\label{sec:charac}
Let $F: \R^{2d} \to \R$ be an averaging operator and let $\mathcal{H}_0$ be as in \eqref{H0}. The main purpose of this section is to show the existence of a multiplicative character in $\mathcal{H}_0$ satisfying certain extremal property. We recall the notation $x\sim y$ whenever $x$,$y\in \Z^d$ are neighbours and $e_{j+d} = -e_j$ for $j=1,\ldots, d$, where $\{e_1,\ldots, e_d\}$ is the canonical basis of $\R^d$.

\begin{prop}\label{Gamma}
Let $F: \R^{2d} \to \R$ be an averaging operator. Define
\begin{equation}\label{Gamma}
\Gamma = \sup \left\{ \frac{u(y)}{u(x)}: u\in \mathcal{H}_0, \,  x,y\in \Z^d,\, x \sim y  \right\}.
\end{equation}
Then $1 \leq \Gamma \leq \frac{1}{\lambda}$ and
\begin{equation}\label{Gammabasis}
\Gamma = \sup \{ v(e_j) : v \in \mathcal{H}_0 \, , \, j= 1,\ldots,2d\}.
\end{equation}
Moreover, $\Gamma$ is attained: there exists $u\in \mathcal{H}_{0}$  and $j\in \{1,\ldots,2d  \}$ such that $\Gamma = u(e_j)$.
\end{prop}

\begin{proof}
That $\Gamma\geq 1$, follows immediately from the definition of $\Gamma$. The fact that $ \Gamma \leq \frac{1}{\lambda}$ follows from the local Harnack inequality, see Proposition~\ref{localH}. In order to prove \eqref{Gamma}, take any two arbitrary neighbouring points $a \sim b$ on $\Z^d$ and suppose that $b-a = e_j$ for some $j\in \{1,..,2d \} $.  If $u\in \mathcal{H}_0$ then it follows from Proposition~\ref{prop:perm-harm} that the function $v$ defined by
$$
v(x) = \frac{u(a + x)}{u(a)}
$$
also belongs to $\mathcal{H}_0$ and
$$
\frac{u(b)}{u(a)} = v(e_j),
$$
for some $j=1,\ldots,2d$. Since $a \sim b$ and are otherwise arbitrary, this establishes \eqref{Gammabasis}.

In order to prove that $\Gamma$ is attained, choose  sequences $(u_n) \subset\mathcal{H}_{0}$  and $(j_n) \subset \{1,\ldots,2d \}$ such that $\displaystyle u_n(e_{j_n}) > \Gamma - 1/n$ for each $n\in \N$. From the Pigeonhole Principle, there exists $j\in \{1,\ldots,2d\} $ so that $j_n = j$ for infinitely many $n$'s so we can assume, without loss of generality, that $j_n = j$ for all $n$. By Proposition \ref{compact}, there exist a subsequence $\{ u_{n_k} \}$ and $u\in \mathcal{H}_{0}$ such that $\{ u_{n_k} \} \to u$ pointwise on $\Z^d$. Then $u(e_j) \geq \Gamma$ so, by definition of $\Gamma$, we must have $u(e_j ) = \Gamma$.
\end{proof}

The following lemma is the key ingredient in the proof of Theorem~\ref{thm:main}. It says that, given an averaging operator $F: \R^{2d} \to \R$,  there exists a  positive $F$-harmonic function on $\Z^d$ which is also a multiplicative character of the form \eqref{charact}.

\begin{lem}\label{lemmain}
Let $F: \R^{2d} \to \R$ be an averaging operator. Then there exist $u \in \mathcal{H}_0$ and $A_1,\ldots,A_d >0$  such that
\begin{equation}\label{mainlem}
u(x_1,\ldots,x_d) = A_1^{ x_1} \cdots A_d^{x_d}
\end{equation}
for each $(x_1,\ldots,x_d) \in \Z^d$.  Furthermore, there is $i\in \{1,\ldots, d \} $ such that, either $A_i = \Gamma $ or $ A_i = \Gamma^{-1}$, where $\Gamma$ is as in \eqref{Gamma}.
\end{lem}

\begin{proof}
By Proposition \ref{Gamma}  we know that $\Gamma_1:=\Gamma \geq 1$  and that
there exist $u_1\in\mathcal{H}_{0}$ and $j_1 \in \{1,\ldots,  2d \}$ such that $u_1(e_{j_1}) = \Gamma_1$. Then we may write
$$
\Gamma_1 = \frac{u_1(e_{j_1})}{u_1(0)} = \frac{F\big(u_1(e_{j_1} +e_1),\ldots, u_1(e_{j_1}+e_{2d})\big )}{F \big (u_1(e_1),\ldots,u_1(e_{2d}) \big )}.
$$
Put $t_j := u_1(e_{j_1
} + e_j)$, $t_j' := u_1(e_j)$ for $j=1,\ldots, 2d$ and observe that
$$
\Gamma_1^{-1} \leq \frac{t_j}{t_j'} \leq \Gamma_1.
$$
for all $j =1,\ldots, 2d$. Therefore
$$
F(t_1,\ldots,t_{2d}) = \Gamma_1 F(t_1',\ldots,t_{2d}') = F(\Gamma_1 t_1',\ldots,\Gamma_1 t_{2d}')
$$
and it follows from Remark~\ref{compar} that $t_j = \Gamma_1 t_j'$ or, equivalently, $u_1(e_{j_1} + e_j ) = \Gamma_1 u_1(e_j)$ for $j =1,\ldots, 2d$. In particular, $u_1(2e_{j_1}) = \Gamma_1 u_1(e_{j_1}) = \Gamma_1^2$ and $u_1(e_{j_1} - e_{j_1}) = 1 = \Gamma_1 u_1(-e_{j_1})$, so $u_1(-e_{j_1}) = \Gamma_1^{-1}$.

Now we can repeat the argument with base points $2e_{j_1}$, $e_{j_1}$:

$$
\Gamma_1 = \frac{u_1(2e_{j_1})}{u_1(e_{j_1})} = \frac{F\big(u_1(2e_{j_1} +e_1),\ldots, u_1(2e_{j_1}+e_{2d})\big )}{F \big (u_1(e_{j_1} +e_1),\ldots, u_1(e_{j_1} + e_{2d}) \big )}
$$
and, as above, we obtain that $u_1 (2e_{j_1} + e_j) = \Gamma_1 u_1 (e_{j_1} + e_j)$. In particular, $u_1(3e_{j_1}) = \Gamma_1 u_1(2e_{j_1}) = \Gamma_1 ^3$. Continuing in the same way we would obtain that $u_1(k e_{j_1}) = \Gamma_1^k $ for each $k\in \N$.
We can also apply the same argument in the opposite direction. Since
$$
\Gamma_1 = \frac{u_1(0)}{u_1(-e_{j_1})} = \frac{F\big(u_1(e_1),\ldots, u_1(e_{2d})\big )}{F \big (u_1(-e_{j_1} +e_1),\ldots, u_1(-e_{j_1} +e_{2d}) \big )}
$$
we may deduce, invoking again Remark~\ref{compar}, that $\displaystyle \Gamma_1 u_1(-e_{j_1} +e_j) = u_1(e_j) $. In particular, $ \Gamma_1^{-1} = u_1 (-e_{j_1} ) = \Gamma_1 u_1(-2e_{j_1})$ so $u_1(-2e_{j_1}) = \Gamma_1^{-2}$ and, analogously $u_1(-ke_{j_1}) = \Gamma_1^{-k}$ for each $k\in \N$.

If $j_1 \in \{1, \ldots, d  \}$ we deduce that $u_1 (0, \ldots, 0, x_{j_1},0, \ldots, 0) = \Gamma_1^{x_{j_1}}$ and if $j_1 \in \{d+1, \ldots, 2d \} $ then $u_1 (0, \ldots, 0, x_{j_1-d},0, \ldots, 0) = \Gamma_1 ^{- x_{j_1 -d}}$. Summarizing, if
$$
i_1 =
\begin{cases}
j_1 \, ,  & \, \text {if} \, \,  j_1 \in \{1, \ldots , d \} \\
j_1 -d \, , & \, \text{if} \, \, j_1 \in  \{d+1, \ldots , 2d \}
\end{cases}
$$
then
$$
u_1 (0, \ldots , 0, x_{i_1}, 0, \ldots , 0) = B_1^{x_{i_1}}
$$
for each $x_{i_1} \in \Z$, where $B_1 = \Gamma_1$ if $j_1 \leq d$ and $B_1 = \Gamma_1^{-1}$ if $j_1 > d$. So far we have restricted ourselves to the level $x_1= \cdots = x_{i_1 -1}= x_{i_1 +1} = \cdots = x_{2d} = 0 $ but we could perform the same argument at any other level and obtain that

\begin{equation}\label{iter1}
u_1 (x_1,\ldots, x_{2d}) = u_1 (x_1,\ldots, x_{i_1 -1},0,x_{i_1 +1} \ldots, x_d ) B_1^{x_{i_1}}.
\end{equation}

This is the first step of the iteration. Now we define $\mathcal{H}_1$ as the subclass of $\mathcal{H}_0$ consisting of those functions $u\in \mathcal{H}_0$ of the form \eqref{iter1} and let further
  $$
  \Gamma_2:=\sup\left\{ \frac{u(y)}{u(x)}: \,  u\in \mathcal{H}_1 \, , \, x\sim y \, , \, x_{i_1} = y_{i_1} = 0 \right\}.
  $$
Now we can argue as in the proof of Proposition~\ref{Gamma} to deduce that there exist  $u_2 \in \mathcal{H}_1$ and $j_2 \in \{1, \ldots , 2d\} \setminus \{i_1, i_1 +d \} $ so that $\Gamma_2 :=  u_2 (e_{j_2})$. Set
$$
i_2 =
\begin{cases}
j_2 \, ,  & \, \text {if} \, \,  j_2 \in \{1, \ldots , d \} \\
j_2 -d \, , & \, \text{if} \, \, j_1 \in  \{d+1, \ldots , 2d \}
\end{cases}
$$
and let $B_2 = \Gamma_2$ if $i_2 \leq d$ and $B_2 = \Gamma_2^{-1}$ if $i_2 >d$. Then, by repeating the argument above we obtain
\begin{equation}\label{iter2}
u_2 (x_1, \ldots, x_d)= u_2 (x_1, \ldots, 0,\ldots, 0, \ldots, x_d ) B_1 ^{x_{i_1}}B_2^{x_{i_2}}
\end{equation}
where the $0$'s in the right hand side of \eqref{iter2} are located  at entries $i_1, i_2$. This finishes the second step of the iteration. By continuing in the same way and, after $d$ steps, we find a permutation $\{ i_1 ,\ldots , i_d  \}$ of $\{ 1, \ldots , d \}$ and positive constants $B_1, \ldots, B_d$ so that
$$
u(x_1, \ldots, x_d ) = B_1^{x_{i_1}}\cdots B_d^{x_{i_d}}
$$
for each $x = (x_1,\ldots, x_d ) \in \Z^d$. Now, putting $A_{i_1}:= B_1, \ldots,  A_{i_d}:= B_d$ we finally get
$$
u(x_1,\ldots, x_d ) = A_1^{x_1}\cdots A_d^{x_d}
$$
and the lemma follows.
\end{proof}

\section{Proof of theorems \ref{thm:main} and \ref{prop:inftylap}}\label{sec:Liouv}

In this section we prove Theorem~\ref{thm:main}, the strong Liouville property for symmetric averaging operators. As we pointed out in the introduction, the theorem applies to some distinguished examples of averaging operators  presented in Section~\ref{sec:ex}, see Corollaries~\ref{cor6.2} and~\ref{cor6.1}. Finally, in Theorem~\ref{prop:inftylap} we show that the discrete $\infty$-Laplacian also satisfies the strong Liouville property, even though the proof of Theorem~\ref{thm:main} fails in that case. Moreover we add a comment about the non symmetric case and end the section with some questions.

%Moreover, we give an example showing that Theorem~\ref{thm:main} may fail for non symmetric weighted discrete Laplacians. We end the section by stating an open problem.

\begin{proof}[Proof of Theorem~\ref{thm:main}]
Once the averaging operator $F$ has been fixed, it is enough to show that $\Gamma = 1$, where $\Gamma $ is as in \eqref{Gamma}.
By Lemma~\ref{lemmain} we obtain a positive $F$-harmonic function $u\in \mathcal{H}_0$ which is a multiplicative character given by~\eqref{mainlem}, where $A_1,\ldots ,A_d $ are positive and, either  $A_i = \Gamma $ or $ A_i = \Gamma^{-1}$ for some $i \in \{ 1, \ldots, d \}$. Therefore,
$$
1 = u(0) = F(u(e_1),\ldots, u (e_{2d})) = F(A_1,\ldots, A_d, A_1^{-1},\ldots, A_d^{-1} )
$$
Since $F$ is symmetric, we infer from  Proposition \ref{symme}, part (3) that $A_1=\cdots =A_d = 1$ and, in particular,  $\Gamma = 1$. This finishes the proof of the theorem.
\end{proof}

%\subsection*{The case of the discrete $\infty$-Laplacian}\label{sec:infty-lap}

As for  discrete $\infty$-harmonic functions, recall that $F_\infty$ is not an averaging operator, since property i) in Definition ~\ref{def:avoper} and the observation in Remark~\ref{compar} both fail. Therefore the proof of Theorem~\ref{thm:main} cannot be directly applied. Nevertheless, as announced in the introduction, the strong Liouville property still holds, see the statement of Theorem~\ref{prop:inftylap}.

\begin{rem}\label{weak}
When checking in detail the proofs of Propositions \ref{prop:perm-harm}, \ref{localH}, \ref{prop:conv}, \ref{compact} and \ref{Gamma}
we observe that only the fact that $F$ is \textit{non decreasing} in each variable together with property ii) in Definition ~\ref{def:avoper} have been used. This implies in particular that Proposition~\ref{Gamma} also holds for the $F_{\infty}$- operator even though $F_{\infty}$ fails to satisfy property i) in Definition ~\ref{def:avoper}.
\end{rem}

\begin{proof}[Proof of Theorem~\ref{prop:inftylap}]

By Remark ~\ref{weak}, we know that Proposition~\ref{Gamma} holds for $F_\infty$ as well, providing a positive $\infty$-harmonic function  $u$ on $\Z^d$ such that $u(0,\ldots,0)=1$ and $\Gamma=u(e_j)$ for some $j\in \{ 1,\ldots, 2d \}$. Then, as in the proof of Theorem~\ref{thm:main}, it is enough to show that $\Gamma=1$. In order to do so, set $t_i:=u(e_i)$ for $i=1,\ldots, 2d$. Then $t_j = \Gamma$ and  $\Gamma^{-1}\leq t_i\leq \Gamma$ for $i=1,\ldots, 2d$. By the definition of the discrete $\infty$-Laplacian we have that
\[
 u(0,\ldots,0)=1=\frac12\left(\max_{1 \leq i \leq 2d} t_i + \min_{1 \leq i \leq 2d}t_i \right)\geq \frac12\left(\Gamma+\Gamma^{-1}\right),
\]
which implies that $\Gamma=1$.
\end{proof}

\begin{rem}\label{symmexa}
It is well known that symmetry is a necessary hypothesis for discrete Liouville-type results (see e.g. Theorem 7.1 in~\cite{S}). We remind the following simple example on $\Z$ showing that
 the strong Liouville property may fail for (non symmetric) weighted discrete Laplacians, even if $d=1$. Indeed, take $R>0$ and define
 $$
 F(t_1, t_2) = \frac{1}{R+1}t_1 + \frac{R}{R+1}t_2.
 $$
Then it is easy to check that $u(x) = R^x$ is $F$-harmonic and positive on $\Z$. Note that $F$ is an averaging operator but, unless $R=1$, it is not symmetric. $F$ is an example of a non symmetric weighted discrete Laplacian.
\end{rem}

{\bf Some questions}.
\begin{enumerate}
\item A limitation of our setting is that averaging operators do not take into account ``spatial dependence", that is, the function $F$ is fixed and the same for all  points $x\in \Z^d$. It would be desirable to obtain counterparts of our results for spatially dependent averaging operators.
\item It would also be desirable to obtain a general result enclosing both Theorems~\ref{thm:main} and~\ref{prop:inftylap}.
\end{enumerate}

\bigskip

\bigskip

{\small

%{\bf Declarations}

{\bf Conflict of Interests} The authors declare that they have no conflict of interest.

{\bf Data Availability Statement} Data sharing not applicable to this article as no datasets were generated or analysed during the current study.}

\end{document}